\documentclass[12pt,a4paper]{amsart}
\usepackage{amsmath, amssymb, amsfonts, amsthm, float, stmaryrd, epsfig,tabularx, mathrsfs}
\usepackage[abbrev,alphabetic]{amsrefs}
\usepackage[pdfborder={0 0 0 [0 0 ]},bookmarksdepth=3, bookmarksopen=true]{hyperref}
\usepackage[latin1,utf8]{inputenc}
\usepackage[T1]{fontenc}
\usepackage{tikz-cd}
\usepackage[minimal]{yhmath}
\usetikzlibrary{arrows}
\usepackage{adjustbox}

\usepackage[capitalize]{cleveref}
\usepackage{color}
\usepackage[ps,all,arc,rotate]{xy}
\usepackage{bbm} 

\usepackage{stackengine,scalerel}
\stackMath

\usepackage{booktabs}

\hypersetup{
	colorlinks=true,
	linkcolor=black,
	citecolor=black,
	filecolor=black,
	urlcolor=black,
}

\numberwithin{figure}{section}
\numberwithin{table}{section}

\theoremstyle{plain}
\newtheorem{thm}{Theorem}[section]
\crefname{thm}{Theorem}{Theorems}
\newtheorem*{prop*}{Proposition}
\newtheorem*{thm*}{Theorem}
\newtheorem{prop}[thm]{Proposition}
\crefname{prop}{Proposition}{Propositions}
\newtheorem{lem}[thm]{Lemma}
\crefname{lem}{Lemma}{Lemmata}

\crefname{cor}{Corollary}{Corollaries}

\newtheorem{conj}[thm]{Conjecture}
\crefname{conj}{Conjecture}{Conjectures}
\crefname{equation}{Equation}{Equations}

\crefname{thmx}{Theorem}{Theorems}

\theoremstyle{definition}

\newtheorem{definition}[thm]{Definition}
\newtheorem*{dfn*}{Definition}

\theoremstyle{remark}

\begin{document}

\title{Virtual fibring of manifolds and groups}
    \author{Dawid Kielak}

\maketitle







\begin{abstract}
	The topic of this survey is the phenomenon of fibring over the circle for manifolds, and its group-theoretic twin, algebraic fibring. We will discuss the state of the art, and explain briefly some of the ideas behind the more recent developments, focusing on RFRS groups and manifolds with such fundamental groups. Then we will move on to a more speculative part, where many conjectures about fibring in higher dimensions will be given. The conjectures vary in their level of plausibility, but even the boldest of them might share the fate of Thurston's Virtually Fibred Conjecture, about which Thurston famously said: ``This dubious-sounding question seems to have a definite chance for a positive answer''.
\end{abstract}


\section*{Structure of the survey}

We will start by introducing fibring over the circle for manifolds in the smooth category. This will be quickly followed by results of Stallings and Farrell that, under certain conditions, reduce the problem to studying finiteness properties of kernels of epimorphisms from the fundamental groups of our manifolds to the integers $\mathbb Z$. We then discuss such finiteness properties, focusing on the homological types $\mathtt{FP}_n(R)$, and bring in the BNS invariants and Novikov rings into the story.

In \Cref{sec virt fibring}, we look into virtual fibring, the main topic. We discuss two types of obstructions: Euler characteristic and $L^2$-homology. This is followed by an introduction to the algebraic viewpoint on $L^2$-Betti numbers, and finally we outline how, for RFRS groups, Novikov and $L^2$-theory interact.

\Cref{sec apps} is devoted to two types of applications of the theory. The first is a weak form of fibring, where kernels are required to have cohomological dimension lower than the original groups; the second revolves around finding witnesses to the lack of (higher) coherence.

Finally, in \Cref{sec higher dim}, we concentrate on high dimensions. First we show how far one can get with the current methods in proving virtual fibring for Poincar\'e-duality groups, and then we speculate about the bigger picture for hyperbolic manifolds.

The main results and conjectures are summarised in \cref{table}.

\begin{table}[htbp]
	\footnotesize
	\centering
	\setlength{\tabcolsep}{6pt}
	\renewcommand{\arraystretch}{1.2}
	\begin{tabularx}{\textwidth}{@{}p{0.52\textwidth} X@{}}
		\toprule 
		\textbf{Hypotheses } & \textbf{Conclusion } \\
		\midrule \midrule
		
		\textbf{Stallings (\cref{Stallings}).} $M$ compact connected $3$–manifold; $G=\pi_1(M)$; $\phi \colon G \twoheadrightarrow \mathbb Z$.  &
		$M$ fibres  with $f_*=\phi$ iff $\ker\phi$ is finitely generated. \\
\addlinespace
			
		\textbf{Farrell (\cref{Farrell}).} $M$ closed connected smooth manifold, $\dim M\geqslant 6$; $G=\pi_1(M)$; $\phi \colon G \twoheadrightarrow \mathbb Z$. &
		$M$ fibres  with $f_*=\phi$ iff (i) the cover for $\ker\phi$ is finitely dominated; (ii) $c(\phi)=0$; (iii) $\tau(\phi)=0$. \\
\addlinespace
		\textbf{Farrell (\cref{Farrell aspherical}).} $M$ closed aspherical smooth, $\dim M\geqslant 6$; $G=\pi_1(M), \mathrm{Wh}(G)=0$; $\phi \colon G \twoheadrightarrow \mathbb Z$. &
		$M$ fibres with $f_*=\phi$ iff $\ker\phi$ is of type $\mathtt{F}$. \\
	\midrule 

		\textbf{L\"uck Mapping Torus Theorem (\cref{Lueck}).} $X$ is a mapping torus with compact base. &
		$\mathrm{H}_\ast^{(2)}(\widetilde X) = 0$. \\
\addlinespace		
				\textbf{Generalised Mapping Torus Theorem (\cref{HennekeKielak}).} $\mathbb K$ a field; $K$ of type $\mathtt{FP}_n(\mathbb K); G = K \rtimes \mathbb Z$. &
		$\beta_n^{(2)}(G;\mathbb K) = 0$ \textrm{ if defined}. \\
			\midrule
	
			\textbf{Agol (\cref{Agol}).} $M$ orientable irreducible compact $3$–manifold; $G=\pi_1(M)$  RFRS. &
	$M$ is virtually fibred iff $\chi(M)= 0$. \\
	\addlinespace	
		\textbf{Agol \& Lott--L\"uck (\cref{LottLueck}).} $M$ compact $3$–manifold; $G=\pi_1(M)$  RFRS. &
		$M$ is virtually fibred iff \(\beta^{(2)}_n(G)=0\) for all \(n\). \\
\addlinespace

		\textbf{Fisher (\cref{Fisher}).} $\mathbb K$ a field; $G$ RFRS of type $\mathrm{FP}_n(\mathbb K)$. &
		 $G$ is virtually $\mathtt{FP}_n(\mathbb K)$-fibred iff \(\beta^{(2)}_i(G;\mathbb K)=0\) for all \(i\leqslant n\).\\

		\midrule	
		
		\textbf{Fisher--Italiano--K.  (\cref{PD thm}).} $G$  RFRS $\mathrm{PD}^n$–group. &
		$G$ is virtually $K \rtimes \mathbb Z$ with $K$ orientable $\mathrm{PD}^{n-1}(\mathbb K)$ over all fields $\mathbb K$ iff $\beta_i^{(2)}(G; \mathbb K) = 0$ for all $i$ and all $\mathbb K$. \\
	\addlinespace
			
	\textbf{Fisher--Italiano--K.  (\cref{PD single K thm}).} $\mathbb K$ a field; $G$  RFRS orientable $\mathrm{PD}^n(\mathbb K)$–group. &
$G$ is virtually $K \rtimes \mathbb Z$ with $K$ orientable $\mathrm{PD}^{n-1}(\mathbb K)$ iff $\beta_i^{(2)}(G; \mathbb K) = 0$ for all $i$. \\
\midrule

\textbf{\cref{PD conj}.} $G$  RFRS $\mathrm{PD}^n$–group. &
$G$ is virtually $K \rtimes \mathbb Z$ with $K$  $\mathrm{PD}^{n-1}$-group iff $\beta_i^{(2)}(G; \mathbb K) = 0$ for all $i$ and all fields $\mathbb K$. \\
\addlinespace

\textbf{\cref{manifold conj}.} $M$ closed aspherical manifold; $G = \pi_1(M)$  RFRS. &
$M$ is virtually fibred iff $\beta_i^{(2)}(G; \mathbb K) = 0$ for all $i$ and all fields $\mathbb K$. \\
\addlinespace

\textbf{Dodziuk \& Gaboriau (\cref{Dodziuk}).} $M$ odd-dimensional hyperbolic manifold of finite volume; $G = \pi_1(M)$. &
$\beta_i^{(2)}(G; \mathbb Q) = 0$ for all $i$. \\
\addlinespace

\textbf{\cref{Dodziuk conjecture}.} $M$ odd-dimensional hyperbolic manifold of finite volume; $G = \pi_1(M)$ RFRS. &
$\beta_i^{(2)}(G;\mathbb K) = 0$ for all $i$ and all fields $\mathbb K$. \\
\addlinespace

\textbf{\cref{RFRS conj}.} $M$ odd-dimensional hyperbolic manifold of finite volume; $G = \pi_1(M)$  RFRS. &
$M$ is virtually fibred. \\
\addlinespace

\textbf{\cref{audacious conj}.} $M$ odd-dimensional hyperbolic manifold of finite volume. &
$M$ is virtually fibred. \\
		\bottomrule
	\end{tabularx}
	\caption{Summary of the main theorems and conjectures.} \label{table}
\end{table}

\section{Fibring}
\label{sec fibring}
\subsection{Fibring of manifolds}
A very general principle that one often uses is to try to understand a complicated object in terms of its constituent building blocks and the interactions between them. In the realm of manifolds, one such technique is to look for fibrations.

\begin{definition}
	Let $M$ and $B$ be smooth  manifolds. A smooth map $f \colon M \to B$ is a \emph{fibre bundle over $B$} if and only if there exists a smooth manifold $F$ such that for every $x \in B$ there exists an open neighbourhood $U$ of $x$ and a diffeomorphism $U \times F \to f^{-1}(U)$ such that composing this diffeomorphism with $f$ gives the projection onto the first factor $U \times F \to U$.
\end{definition}

The manifold $B$ is the \emph{base}, and $F$ the \emph{fibre}; it is immediate that $\dim M = \dim B + \dim F$. Throughout the survey, we are interested only in $B$ being the circle $\mathbb{S}^1$. Hence we will never specify the base, and it will always be $\mathbb S^1$. Moreover, as is common in low-dimensional topology, we will call the map $f$ above a \emph{fibration}, and we will say that $M$ \emph{fibres} if and only if it admits such a fibration.

If a  smooth  $n$-manifold $M$ fibres  with fibre $F$, then $M$ is diffeomorphic to the \emph{mapping torus} of a diffeomorphism $m \colon F\to F$ (known as the \emph{monodromy}), where the mapping torus is defined to be
 \[F \times [0,1]/\!\!\sim\]
 with $\sim$ being the finest equivalence relation satisfying $(p,0) \sim (m(p),1)$ for every $p \in F$. This reduces the task of understanding $M$ to that of understanding the pair $(F,m)$, which can be thought of as a dynamical system.

A good example of using this technique to great advantage in practice is the study of compact $3$-manifolds, since we have a deep understanding of both compact $2$-manifolds (surfaces), and their diffeomorphism -- up to smooth isotopy, we can use the Nielsen--Thurston classification of mapping classes of surfaces. We do not lose anything by considering $m$ up to smooth isotopy, since the resulting mapping tori are diffeomorphic. An example of this approach is Thurston's Hyperbolisation Theorem for the fibred case, see \cite{Sullivan1981}.

\begin{thm}[Thurston]
	If $M$ is a closed $3$-manifold that fibres with fibre $F$ and monodromy $m$, then $M$ admits a hyperbolic metric if and only if $m$ defines a pseudo-Anosov mapping class of the surface $F$.
\end{thm}

At this stage, the reader is hopefully convinced that it is desirable to  be able to show that a manifold fibres. But how should one do it? For $3$-manifolds there is a beautiful and simple algebraic condition, given by Stallings \cite{Stallings1961}.

\begin{thm}[Stallings]
	\label{Stallings}
	Let $M$ be a compact connected $3$-manifold. Let $G = \pi_1(M)$ and let $\phi \colon G \to \mathbb{Z}$ be an epimorphism. There exists a fibration $f \colon M \to \mathbb S^1$ with $f_\ast = \phi$ if and only if the kernel of $\phi$ is finitely generated.  
\end{thm}

In his paper, Stallings has two additional assumptions: that $M$ is irreducible, and that $\ker \phi \not\cong \mathbb{Z}/2\mathbb{Z}$. Both are unnecessary, in view of the resolution of the Poincar\'e Conjecture by Perelman \cite{Perelman2002,Perelman2003}.

Stallings's criterion reduces the topological question of fibring (for $3$-manifolds) to a group-theoretic one. In high dimensions, a similar result was proved by Farrell in his thesis \cite{Farrell1967}.

\begin{thm}[Farrell]
	\label{Farrell}
Let $M$ be a closed connected smooth manifold of dimension at least $6$ with fundamental group $G$, and let $\phi \colon G \to \mathbb{Z}$ be an epimorphism. There exists a fibration $f \colon M \to \mathbb S^1$ with $f_\ast = \phi$ if and only if the following three conditions hold:
\begin{enumerate}
	\item the covering space of $M$ corresponding to $\ker \phi$ is finitely dominated,
	\item $c(\phi)=0$, and
	\item $\tau(\phi) = 0$.
\end{enumerate}
\end{thm}

The two obstructions $c$ and $\tau$ are defined in a $K$-theoretic fashion. It will not be terribly important for us what their precise definition is, since it was shown by Farrell--Hsiang \cite{FarrellHsiang1970}, Theorem 21, and independently  Siebenmann \cite{Siebenmann1970}, that if the Whitehead group $\mathrm{Wh}(G)$ of $G$ vanishes, then $c$ and $\tau$ also take value in the trivial abelian group $0$. 

The theorem is true verbatim also when $M$ has connected boundary that itself fibres over the circle; in this case we need to know that $\phi$ is induced by some smooth map $M \to \mathbb{S}^1$ that restricts to a fibration on the boundary. The situation is particularly straightforward when the boundary is a torus, since then it is enough to know that $\phi$ does not vanish on its fundamental group.

If $\phi$ is induced by a fibration, then $\ker \phi$ is the fundamental group of a compact manifold, and so the condition of being finitely dominated can be changed to being of finite type. If $M$ is aspherical then so is the fibre, and hence we obtain the following special case of Farrell's theorem.

\begin{thm}
	\label{Farrell aspherical}
	Let $M$ be a closed aspherical smooth manifold of dimension at least $6$ with $G = \pi_1(M)$, and suppose that $\mathrm{Wh}(G) = 0$. Let $\phi \colon G \to \mathbb{Z}$ be an epimorphism. There exists a fibration $f \colon M \to \mathbb S^1$ with $f_\ast = \phi$ if and only if $\ker \phi$ admits a finite classifying space.
	\end{thm}

This is now very much in the spirit of Stallings's result --  fibring is encoded in finiteness properties of the kernel. Also, the condition on $\mathrm{Wh}(G)$ vanishing is not very restrictive -- it is conjectured that $\mathrm{Wh}(G)= 0$ for every torsion-free group, and hence in particular for fundamental groups of aspherical manifolds. The vanishing of $\mathrm{Wh}(G)$ has been established for many groups, in particular for all torsion-free groups satisfying the Farrell--Jones conjecture.

\subsection{Fibring of groups}

As we have just seen, the question of whether an aspherical manifold fibres can be, under rather mild conditions, reduced to the study of finiteness properties of kernels of epimorphisms to $\mathbb{Z}$.  Finiteness properties come in two flavours, homotopic and homological.

\begin{definition}
For $n \in \mathbb{N} \cup \{\infty\}$, we say that a group $G$ is of \emph{type $\mathtt{F}_n$} if and only if it admits a classifying space with finitely many cells in dimension $i$ for every $i \leqslant n$. We say that $G$ is of \emph{type $\mathtt{F}$} if and only if it admits a classifying space with finitely many cells altogether.  
	\end{definition}
	
These are the homotopic finiteness properties. We see that \cref{Farrell aspherical} connects fibring to $\ker \phi$ being of type $\mathtt{F}$.

If one wants to compute group homology of $G$, one needs to start with a resolution of the trivial $\mathbb{Z} G$-module $\mathbb{Z}$ by projective modules, that is, one needs an exact sequence
\[
\dots \to C_n \to \dots \to C_1 \to C_0 \to \mathbb{Z}
\]
where every $C_i$ is a projective $\mathbb{Z} G$-module. Since we generally want to have the group acting on the left, we take $C_i$ to be \emph{left} modules. There are canonical constructions of such resolutions, for example the bar resolution, but for efficient computation one wants to take a resolution where the modules are as small as possible. If $G$ is of type $\mathtt{F}_n$, the cellular chain complex of the universal cover of a classifying space with finite $n$-skeleton gives us a resolution where all the modules $C_i$ for $i \leqslant n$ are finitely generated (they are also free, but this is less important here). Similarly, when $G$ is of type $\mathtt{F}$ we will get a resolution by finitely generated modules that is finite, that is, $C_i = 0$ for large enough $i$. There is nothing special about $\mathbb{Z}$ here -- we can run the entire discussion over our favourite base ring.

From homological perspective, this is all that having a small classifying space gives us. This motivates the definition of homological finiteness properties.

\begin{definition}
	Let $R$ be a ring.
For $n \in \mathbb{N} \cup \{\infty\}$, we say that a group $G$ is of \emph{type $\mathtt{FP}_n(R)$} if and only if the trivial $RG$-module $R$ admits a resolution $C_\bullet$ by projective $RG$-modules such that $C_i$ is finitely generated for every $i \leqslant n$.  
 We say that $G$ is of \emph{type $\mathtt{FP}(R)$} if and only if $R$ admits a finite resolution by finitely generated projective $RG$-modules.	
\end{definition}

One writes $\mathtt{FP}_n$ and $\mathtt{FP}$ for $\mathtt{FP}_n(\mathbb{Z})$ and $\mathtt{FP}(\mathbb{Z})$, respectively. We have implications
\[
\mathtt{F}_n \Rightarrow \mathtt{FP}_n \Rightarrow \mathtt{FP}_n(R) 
\]
for every ring $R$. (This is probably a good place to say that our rings are always unital, associative, and non-zero. Ring morphisms take identities to identities.)

For $n=1$, all three properties coincide -- they are all equivalent to the group $G$ being finitely generated. This is important, as it allows for homological proofs of \cref{Stallings}, see \cite{Hillman2020} and \cite{Bridsonetal2023}.

For $n \geqslant 2$, the implications are proper. First examples were given by Bestvina--Brady \cite{BestvinaBrady1997}, and now we have uncountably many groups of type $\mathtt{FP}_2$ that are not $\mathtt{F}_2$ (that is, finitely presented), as constructed by Leary \cite{Leary2018}.

In view of the above, the following definition should feel natural.

\begin{definition}[Algebraic fibring]
	An epimorphism $\phi \colon G \to \mathbb{Z}$ will be called an \emph{$\mathtt{FP}_n(R)$-fibration} if and only if its kernel has type $\mathtt{FP}_n(R)$. The group $G$ will be called \emph{$\mathtt{FP}_n(R)$-fibred} if and only if it admits an $\mathtt{FP}_n(R)$-fibration.
\end{definition}

Sometimes the terminology \emph{algebraic fibring} is used for $\mathtt{FP}_1$-fibring. The author prefers to use  `algebraic fibring' as an umbrella term for  $\mathtt{FP}_n(R)$-fibring for various values of $n$.

\subsection{BNS invariants and the Novikov rings}

In the discussion above, we have already seen the group ring $RG$. As an $R$-module, this is the free $R$-module with basis $G$, and hence elements can be written as finite formal sums of elements of $G$ (the basis) with coefficients in $R$. The multiplication on $G$ can be extended by linearity to $RG$ and makes it into a ring; the copy of $R$ by the identity element $1$ of $G$ turns $RG$ into an $R$-algebra. This construction does not use inverses in $G$, and can equally well be applied to monoids. The usual polynomial ring in one variable and coefficients in $R$ can be viewed this way as the monoid ring $R \mathbb{N}$.

Importantly, homological finiteness properties can be defined in exactly the same way for monoids using monoid rings instead of group rings. 

How to decide what finiteness properties the kernel $\ker \phi$ has for an epimorphism $\phi \colon G \to \mathbb{Z}$? A key insight was an observation of Bieri--Neumann--Strebel that instead of dealing with the kernel, it is much easier to study finiteness properties of monoid rings $R G_\phi$ where 
\[G_\phi = \{g \in G \mid \phi(g) \geqslant 0\},\]
 and doing this for both $\phi$ and $-\phi \colon g \mapsto -\phi(g)$ yields information about $\ker \phi$.

\begin{definition}[BNS invariants]
	Let $R$ be a ring and $G$ a group of type $\mathtt{FP}_n(R)$. The \emph{$n$th homological BNS invariant over $R$}, denoted $\Sigma^n(G;R)$, is the subset of $\mathrm{H}^1(G;\mathbb{R}) \smallsetminus \{0\}$ consisting precisely of those homomorphisms $\phi \colon G \to \mathbb{R}$ for which $G_\phi$ is a monoid of type $\mathrm{FP}_n(R)$.
\end{definition}
	Here, $\mathrm{H}^1(G;\mathbb{R})$ is identified with the set of homomorphisms $G \to \mathbb{R}$ and is endowed with the obvious topology (when $G$ is finitely generated). The monoid $G_\phi$ is defined exactly as for maps to the integers. The definition for $n=1$ was formulated in \cite{Bierietal1987}; for higher $n$ it is due to Bieri--Renz \cite{BieriRenz1988}.
	
	One can also define homotopic BNS invariants, but this is more involved; for an approach using closed $1$-forms see \cite{Farberetal2010}. In fact, one of the challenges of the theory presented here is to find analogues for the discussion that will follow in the homotopic setting.

There are four key properties that the BNS invariants exhibit.

\begin{thm}[\cite{Bierietal1987,BieriRenz1988}]
	\label{BNS}
		Let $R$ be a ring, and $G$ a group of type $\mathtt{FP}_n(R)$ with $n \geqslant 1$.
		\begin{enumerate}
			\item $\Sigma^n(G;R)$ is an open subset of $\mathrm{H}^1(G;\mathbb{R})$.
					\item $\Sigma^n(G;R)$ is invariant under positive homothety $\phi \mapsto \lambda \phi$, $\lambda > 0$.	
			\item For every $S \subseteq \mathrm{H}^1(G;\mathbb{R})$, the group $K = \bigcap_{\phi \in S} \ker \phi$ is of type $\mathtt{FP}_n(R)$ if and only if
			\[
			\{ \psi \in \mathrm{H}^1(G;\mathbb{R}) \mid K \leqslant \ker \psi \} \subseteq \Sigma^n(G;R).
			\]
		\end{enumerate}
\end{thm}

We are not going to be particularly concerned with maps other than epimorphism $G \to \mathbb{Z}$, but the openness property is important, and the above framework makes it more natural to state it.

The introduction of the fourth property requires some more work. Given a group $G$ and an epimorphism $\phi \colon G \to \mathbb{Z}$, we pick $t \in G$ such that $\phi(t) = 1$. Let $\alpha$ be an automorphism of $K = \ker \phi$ defined by $\alpha(k) = tkt^{-1}$. We will use the same symbol to denote the extension of $\alpha$ to an isomorphism of $R$-algebras $RK \to RK$. We  define the \emph{twisted Laurent polynomial ring} with coefficients in $R K$, variable $t$, and twisting $\alpha$, to be the free $R K$-module with basis $\{ t^i \mid i \in \mathbb{Z}\}$. We now endow it with a ring structure by extending by $R$-linearity the rule
\[
xt^i \cdot yt^j = x \alpha^{i}(y)t^{i+j}. 
\]  

It is not hard to see that this ring is isomorphic to $RG$ by the map that sends $RK$ to itself identically, and sends the formal symbol $t^i$ to the element $t^i \in G$.

\begin{definition}[Novikov ring]
We define the \emph{Novikov ring} $\widehat {R G}^\phi$ to be the ring of twisted Laurent series with coefficients in $R K$, variable $t$, and twisting $\alpha$.
\end{definition}

The only difference to the construction above (which yold the group ring $RG$) is that we now allow formal sums of powers of $t$ with coefficients in $RK$ where the powers are all distinct and bounded only from below, rather than bounded from below and above.

\begin{thm}[Sikorav \cite{Sikorav1987}, Fisher \cite{Fisher2024}]
	\label{Sikorav}
	Let $G$ be a group of type $\mathtt{FP}_n(R)$. An epimorphism $\phi \colon G \to \mathbb{Z}$ belongs to $\Sigma^n(G;R)$ if and only if 
	\[
	\mathrm{H}_i(G; \widehat{R G}^\phi) = 0 \textrm{ for all } i \leqslant n.
	\]
\end{thm}

Combining \cref{BNS} and \cref{Sikorav} allows us to check homological finiteness properties of $\ker \phi$ by studying the homology of $G$ with coefficients in the Novikov rings $\widehat{R G}^\phi$ and $\widehat{R G}^{-\phi}$. We will refer to such homologies as \emph{Novikov homology}.

\section{Virtual fibring}
\label{sec virt fibring}
\subsection{Obstructions}

So far, we were focused on deciding whether a given manifold or  group fibres. It is however very useful, both for manifolds and for groups, to exhibit \emph{virtual fibring}, that is, fibring of a finite-index cover or subgroup. The lesson we are learning from $3$-manifolds is that this property is much more common than being fibred on the nose. For example, celebrated results of Agol \cite{Agol2013} (in the closed case) and Wise \cite{Wise2012} (in the cusped case) show that all finite-volume hyperbolic $3$-manifolds are virtually fibred (this confirmed a famous conjecture of Thurston that we alluded to in the abstract). But it is not hard to find such manifolds that are not fibred -- for example, take the complement of the $5_2$ knot (in Rolfsen notation).

There are many more examples of virtually fibred $3$-manifolds. Among orientable prime ones with toroidal boundary, it is only some of the closed graph manifolds that do not carry a metric of non-positive curvature that are not virtually fibred. Being prime is necessary for virtual fibring among orientable $3$-manifolds. 

Studying virtual fibring presents us with a new challenge: even if we could decide for a given manifold or group whether it fibres (and `even if' is needed here, since we typically cannot!), to rule out virtual fibring we need to sift through all finite-index covers or subgroups. Hence it is natural to look for obstructions to virtual fibring.

The first (but surprisingly powerful) invariant is the Euler characteristic $\chi$: it is not hard to see that a compact manifold that fibres must have Euler characteristic equal to zero, and Euler characteristic is multiplicative in the index of a covering. Hence a compact manifold that virtually fibres must have vanishing Euler characteristic. Combining this with the Chern--Gauss--Bonnet theorem we can immediately conclude the following.

\begin{prop}
	Even-dimensional hyperbolic manifolds of finite volume do not virtually fibre.
\end{prop}

In certain situations, the Euler characteristic becomes a perfect obstruction for virtual fibring.

\begin{thm}[Agol \cite{Agol2008}]
	\label{Agol}
	Let $M$ be an orientable  irreducible compact $3$-manifold whose fundamental group has the RFRS property. The manifold is virtually fibred if and only if $\chi(M) = 0$.
\end{thm}
(We will come back to the RFRS property later.)

There is a family of more sophisticated invariants that behave very much like the Euler characteristic, but are even more useful in obstructing virtual fibring -- $L^2$-Betti numbers.

\subsection{$L$\texorpdfstring{$^2$}{\texttwosuperior}-homology}

We will start by briefly discussing the classical theory. Let $X$ be a CW-complex on which a group $G$ acts freely, cocompactly, and cellularly (this last property means that if a cell is fixed setwise, it is fixed pointwise).
Let $C_\bullet$ denote the cellular chain complex of $X$ over $\mathbb C$ (we can also work over $\mathbb R$ if we prefer). Every $C_i$ is a free finitely generated $\mathbb C G$-module.
We may also define a chain complex $C^{(2)}_\bullet$, where the chains in $C^{(2)}_i$ are not necessarily finitely supported, but are $L^2$-summable. The differentials are induced by taking the boundary in the usual way. Another way of building this chain complex is
to tensor $C_\bullet$ with the Hilbert space $\ell^2 (G)$ over $\mathbb{C} G$ (the action on $\ell^2 (G)$ is on the \emph{right}).

 We are now looking at the chain complex 
\[\dots \to  C_{i+1}^{(2)} \to C_{i}^{(2)} \to \cdots .\]
Each term is a Hilbert space. The differentials are linear (acting on the right), their kernels are closed, but their images might not be. We remedy this by looking at the \emph{reduced} (or \emph{Hausdorff}) homology of the chain complex: we divide the kernel of the differential $\partial_i$ by the closure of the image of $\partial_{i+1}$. The homology groups $\mathrm{H}_n^{(2)}(X)$ are now Hilbert spaces themselves, and are known as the $L^2$-homology groups of $X$.

Since orthogonal complements give us canonical sections to quotient maps between Hilbert spaces, we may identify $L^2$-homology groups with closed subspaces of the Hilbert spaces $C_i^{(2)}$. These latter spaces are isomorphic to finite direct sums of copies of $\ell^2(G)$. Since differentials were $G$-equivariant (for the left $G$-action), the $L^2$-homology groups are $G$-invariant subspaces. Such subspaces are known as \emph{Hilbert $G$-modules}.

Each such subspace admits a $G$-equivariant projection from $C_i^{(2)}$ onto itself. The projection is naturally a finite matrix over the group von Neumann algebra $\mathcal N(G)$, the algebra of $G$-equivariant continuous operators on $\ell^2(G)$, and hence admits a \emph{von Neumann trace}, a non-negative real number. The \emph{$n$th $L^2$-Betti number} of the pair $(X,G)$ is precisely the von Neumann trace of the projection corresponding to $\mathrm{H}_n^{(2)}(X)$. In general, the trace of a $G$-equivariant projection is known as the \emph{von Neumann dimension} of the Hilbert $G$-module.

In the special case where $X$ is the universal covering of a compact classifying space for $G$, we talk about the \emph{$n$th $L^2$-Betti number} of $G$ and denote it by $\beta_n^{(2)}(G)$. We will also write $\mathrm{H}_n^{(2)}(G)$ for $\mathrm{H}_n^{(2)}(X)$.

A key property of the von Neumann dimension is that is satisfies the rank-nullity formula. Also, as the trace of the identity operator is $1$, the von Neumann dimension of $\bigoplus_n \ell^2(G)$ is precisely $n$. Combining these two facts immediately shows  that for a group $G$ of type $\mathtt{F}$, the \emph{$L^2$-Euler characteristic $\chi^{(2)}(G)$}, that is, the alternating sum of the $L^2$-Betti numbers, is equal to the alternating sum of the ranks of the free modules $C_i$. This is in turn the same as the alternating sum of the number of cells of each dimension of a finite classifying space of $G$, that is, the Euler characteristic $\chi(G)$. So the $L^2$-homology, just like usual homology with trivial coefficients, gives us dimension-by-dimension information that altogether can be combined into the Euler characteristic. Both homologies are also invariant under homotopy. The advantage of $L^2$-homology (that will be crucial for us here) lies in its behaviour under passing to finite-index subgroups.

If $H$ is a finite-index subgroup of our $G$, then the resolution $C_\bullet$ becomes a resolution of the trivial $\mathbb{Z} H$-module $\mathbb{Z}$ simply by restriction: every $\mathbb{Z} G$-module $C_i$ is now treated as a $\mathbb{Z} H$-module. As $|G : H| < \infty$, the Hilbert spaces $\ell^2(G) \otimes_{\mathbb{Z} G} C_i$ and $\ell^2(H) \otimes_{\mathbb{Z} H} C_i$ are isomorphic, in fact in an $H$-equivariant way (the group $H$ acts on the left here). The bottom line is that $\mathrm{H}_n^{(2)}(G)$ and $\mathrm{H}_n^{(2)}(H)$ are isomorphic as Hilbert $H$-modules. This does not mean that the $L^2$-Betti numbers are the same: $\beta_n^{(2)}(G)$ is computed as the von Neumann trace of a $G$-equivariant projection, a $k\times k$ matrix over the group von Neumann algebra of $G$, where $C_n = \bigoplus_k \mathbb{Z} G$. The same projection can be understood as a $k|G:H| \times k|G:H|$ matrix over the group von Neumann algebra of $H$, and an easy calculation shows that the von Neumann trace of the projection over $H$ is equal to that over $G$ multiplied by the index $|G:H|$. To conclude, we have $\beta_n^{(2)}(H) = |G:H| \beta_n^{(2)}(G)$. This is not something that we can expect from the usual Betti numbers.

Here is the upshot: vanishing of $\beta_n^{(2)}$ is a property stable under passing to finite-index subgroups. It turns out that it is also an obstruction to fibring, and hence to virtual fibring. 

\begin{thm}[L\"uck Mapping Torus Theorem \cite{Lueck1994a}]
	\label{Lueck}
	If $Y$ is a compact connected  CW-complex and $X$ is a mapping torus of a cellular homeomorphism $m \colon Y \to Y$, then the universal cover of $X$ is $L^2$-acyclic.
\end{thm}

In particular, if a compact manifold virtually fibres then it has to be $L^2$-acyclic.

Recall \cref{Agol}: for irreducible $3$-manifolds with RFRS fundamental groups, the Euler characteristic is a perfect obstruction to virtual fibring. Combining this with the computations of $L^2$-homology of $3$-manifolds by Lott--L\"uck \cite{LottLueck1995}, we obtain a sharper result.

\begin{thm}
	\label{LottLueck}
	Let $M$ be a compact  $3$-manifold whose fundamental group $G$ has the RFRS property. The manifold is virtually fibred if and only if $\beta_n^{(2)}(G) = 0$ for all $n$.
\end{thm}

The difference between \cref{LottLueck} and \cref{Agol} is the assumption of the manifold being irreducible (orientability does not really play a role in virtual fibring). If a closed connected orientable $3$-manifold is not irreducible, then either it is $\mathbb S^2 \times \mathbb S^1$, or its first $L^2$-Betti number is positive. This is a nice concrete manifestation of $L^2$-homology being more powerful than Euler characteristic in obstructing virtual fibring. 

The construction of $L^2$-homology can be generalised in various ways, for example by going away from type $\mathtt{F}$, or even type $\mathtt{FP}_\infty$, and allowing infinity as a possible value for $\beta_n^{(2)}(G)$. A detailed discussion of this any many other related topics can be found in L\"uck's book \cite{Lueck2002}. This sometimes allows us to define $\chi^{(2)}$ in situations where $\chi$ cannot be defined; one natural place where this is used is the study of the Thurston norm for $3$-manifolds and free-by-cyclic groups, see \cite{FriedlLueck2019a,FunkeKielak2018}.

The von Neumann dimension of a Hilbert $G$-module can be any non-negative real number. In the context of $L^2$-homology, the modules that appear are not arbitrary however -- they come from kernels and cokernels of the differentials, and these can be thought of as matrices over the group ring $\mathbb{Z} G$. It is expected that this is a serious restriction. One way of making such an expectation precise is the following.

\begin{conj}[Atiyah conjecture, torsion-free case]
	Let $G$ be a torsion-free group and let $A$ be a finite $n\times m$ matrix over $\mathbb{Z} G$. The von Neumann dimension of the kernel of the map
	\[
	\bigoplus_n \ell^2(G) \to \bigoplus_m \ell^2(G), \ x \mapsto xA  
	\]
	is an integer.
\end{conj}

The conjecture is open; it is known for extensions of free groups by  elementary amenable groups (Linnell \cite{Linnell1993}), more generally for residually \{torsion-free elementary amenable\} groups (Schick \cite{Schick2000,Schick2002}); it passes to subgroups, and it is stable under free products when the factors are countable (S\'anchez-Peralta \cite{Sanchez-Peralta2024}). We do not currently have any geometric means of establishing the conjecture. In particular, it is a very important open problem to settle the Atiyah conjecture for torsion-free hyperbolic groups.

\subsection{The Linnell skew-field}

We have already seen that $L^2$-homology  obstructs virtual fibring. To be able to generalise \cref{LottLueck}, we will need to introduce a different viewpoint, discovered by Linnell in his work on the Atiyah conjecture for free groups.

We have already met the group von Neumann algebra $\mathcal N(G)$, whose elements are continuous $G$-equivariant operators on $\ell^2(G)$. A way to think about it is to view $\mathcal N( G)$ as acting on the left; it then becomes clear that $\mathbb{Q} G$ acting on the left on $\ell^2(G)$ is a subring of $\mathcal N(G)$.
It was shown by L\"uck \cite{Lueck1998} that one can build the theory of $L^2$-homology completely algebraically, as group homology with coefficients in $\mathcal N(G)$. Instead of making the homology reduced, as we had to do when working with Hilbert spaces, one divides the homology groups by their $\mathcal N(G)$-torsion (in a suitable sense), and defines the von Neumann dimension of the resulting module to be the von Neumann trace of the corresponding projection, as before. 

 From the ring-theoretic perspective, one of the key properties of $\mathcal N(G)$ is that the set $\mathcal S$ of non-zero divisors in $\mathcal N(G)$ satisfies the \emph{(two-sided) Ore condition}: $q \mathcal N(G) \cap p \mathcal S \neq \emptyset$ and $ \mathcal N(G)q \cap  \mathcal S p \neq \emptyset$ for every $p \in \mathcal N(G), q \in \mathcal S$.  This may seem mysterious, but all it says is that a `right fraction' $pq^{-1}$ can be turned into a `left fraction' ${q'}^{-1}p'$ and vice-versa, where the denominators $q$ and $q'$ lie in $\mathcal S$, and numerators $p$ and $p'$ in $\mathcal N(G)$.  

Equipped with the Ore condition, we can build the ring of fractions $\mathcal U(G)$ of $\mathcal N(G)$ with denominators in $\mathcal S$.
The ring $\mathcal U(G)$ is known as the algebra of operators \emph{affiliated with $\mathcal N(G)$}, and admits also a functional-analytic description. Formally, the ring of fractions is a localisation (the Ore localisation, to be more specific), and since localisation is flat, we obtain
\[
\mathrm H_n(G;\mathcal U(G)) = \mathrm H_n(G;\mathcal N(G)) \otimes_{\mathcal N(G)} \mathcal U(G). 
\]
Tensoring with $\mathcal U(G)$ rids us of the $\mathcal N(G)$-torsion issue, and one can again build a theory of von Neumann dimension for $\mathcal U(G)$-modules that gives the same $L^2$-Betti numbers.

We passed from $\mathcal N(G)$ to a bigger ring, and now we are going to find a smaller one that fits our bill.

\begin{definition}
	Given an inclusion of rings $S \leqslant T$, we say that $S$ is \emph{division closed} in $T$ if and only if every element in $S$ that admits a two-sided inverse in $T$ admits such an inverse in $S$. Given rings $R \leqslant T$, the \emph{division closure} of $R$ in $T$ is the intersection of all division-closed subrings of $T$ that contain $R$.
	
	The division closure of $\mathbb{Q} G$ in $\mathcal U(G)$, denoted $\mathcal D_{\mathbb{Q} G}$, is the \emph{Linnell ring} of $G$.
\end{definition}

Warning: Andrei Jaikin-Zapirain talks about \emph{Linnell} skew-fields, where being Linnell is a special (desirable) property. The Linnell ring in our sense is a Linnell skew-field in Jaikin's, provided that it is a skew-field.

\begin{thm}[Linnell \cite{Linnell1993}, see also Lemma 10.39 of \cite{Lueck2002}]
	A torsion-free group $G$ satisfies the Atiyah conjecture if and only if $\mathcal D_{\mathbb{Q} G}$ is a skew-field.
\end{thm}

When $\mathcal D_{\mathbb{Q} G}$ is a skew-field, then $\mathcal U(G)$ is a flat $\mathcal D_{\mathbb{Q} G}$-module, since all modules over skew-fields are flat (by basic linear algebra), so
\[
\mathrm H_n(G;\mathcal U(G)) = \mathrm H_n(G;\mathcal D_{\mathbb{Q} G}) \otimes_{\mathcal D_{\mathbb{Q} G}} \mathcal U(G). 
\]
Let $\beta_n^{(2)}(G;\mathbb{Q}) = \dim_{\mathcal D_{\mathbb{Q} G}} \mathrm{H}_n(G;\mathcal D_{\mathbb{Q} G})$.
Since the von Neumann dimension of $\mathcal U(G)$ is $1$, we see that  $\beta_n^{(2)}(G;\mathbb{Q}) = \beta_n^{(2)}(G)$. We introduce this extra piece of notation since it remembers the ground field, and this will become important later.

The groups that we will consider will all be torsion free and will satisfy the Atiyah conjecture. We will work with $L^2$-homology exclusively in its incarnation as $\mathrm{H}_n(G;\mathcal D_{\mathbb Q G})$.

The skew-field $\mathcal D_{\mathbb Q G}$ behaves extremely well with respect to subgroups: given a pair $H \leqslant G$, with $G$ torsion-free and satisfying the Atiyah conjecture, and a set of right transversals $T$ for $H$ in $G$ with $1 \in T$, we have embeddings of left $\mathcal D_{\mathbb Q H}$-modules
\[
\mathcal D_{\mathbb Q H} \to \bigoplus_{t \in T} \mathcal D_{\mathbb Q H}t \to \mathcal D_{\mathbb Q G},
\]
where the first map is $x \mapsto x\cdot 1$, and the second one is induced by $x \cdot t \mapsto x t$. This property is known as \emph{strong Hughes-freeness} or the \emph{Linnell property}. Furthermore, if $H$ is a finite-index subgroup then the second map is onto. 

\subsection{RFRS}

We have seen how Novikov homology controls fibring for groups, and how $L^2$-homology obstructs virtual fibring. It is time to bring the two strands together. So far, the most general framework under which the theories are known to communicate with each other is that of RFRS groups.

\begin{definition}[Agol \cite{Agol2008}]
	A group $G$ is \emph{residually finite rationally solvable}, or \emph{RFRS}, if and only if there is a chain of finite-index normal subgroups $G = G_0 \geqslant G_1 \geqslant \dots$ such that
	\begin{enumerate}
		\item $\bigcap_i G_i = \{1\}$, that is, the chain is \emph{residual}, and
		\item for every $i$, the group $G_{i+1}$ contains the kernel of the natural map $G_i \to \mathrm{H}_1(G_i;\mathbb{Q})$.
	\end{enumerate}
\end{definition}

Since we are only interested in finitely generated groups, it is worth mentioning an equivalent definition in this setting.

\begin{thm}[Theorem 6.3 in \cite{OkunSchreve2025}]
	\label{OkunSchreve}
	A finitely generated group is RFRS if and only if it is residually \{virtually abelian and locally indicable\}.
\end{thm}

Recall that a group is \emph{locally indicable} if and only if its every non-trivial finitely generated subgroup maps onto $\mathbb{Z}$. It is not hard to show that finitely generated virtually abelian locally indicable groups are precisely the virtually abelian poly-$\mathbb{Z}$ groups, or the diffuse Bieberbach groups.

The class of RFRS groups is surprisingly large: one can show directly that all right-angled Artin groups (RAAGs) are RFRS (Agol in \cite{Agol2008} shows only that they are virtually RFRS, going via right-angled Coxeter groups). It is immediate from the definition that being RFRS passes to subgroups, and so all subgroups of RAAGs are RFRS. In particular, all groups that are special in the sense of Haglund--Wise have this property.

Also, finitely generated RFRS groups satisfy the Atiyah conjecture, since they are residually \{torsion-free elementary amenable\} by  \cref{OkunSchreve}. We have therefore the skew-field $\mathcal D_{\mathbb Q G}$ at our disposal. Moreover, since RFRS groups are locally indicable themselves, their group rings have only trivial units, that is, the only invertible elements are of the form $\lambda g$ where $g \in G$ and $\lambda$ is a unit of the ground ring.

It was a conjecture of Kaplansky that torsion-free groups should only have such trivial units. This has been shown to be false by Gardam \cite{Gardam2021}. We do not have a clear picture of where to expect non-trivial units to manifest themselves.

For RFRS groups, it turns out that one can relate  $\mathcal D_{\mathbb Q G}$ and the Novikov rings of finite-index subgroups. The construction is quite technical, so let us explain it only for an example. Let $x \in \mathbb Q G$ be a non-zero element. Since $\mathcal D_{\mathbb Q G}$ is a skew-field containing $\mathbb Q G$, our $x$ is invertible in $\mathcal D_{\mathbb Q G}$. Our first attempt would be to show that it is also invertible in $\widehat{\mathbb Q G}^\phi$ for some epimorphism $\phi \colon G \to \mathbb Z$. Let $G^\mathrm{fab}$ denote the free part of the abelianisation of $G$; this is precisely the image of $G$ in $\mathrm{H}_1(G;\mathbb Q)$ that appears in the definition of RFRS. We denote the natural epimorphism $G \to G^\mathrm{fab}$ by $\alpha$.

The element $x$ has a \emph{support} $\mathrm{supp} \, x$ in $G$, namely the smallest subset such that $x$ is a $\mathbb Q$-linear combination of elements from this subset. Using the fact that $\mathbb Q G$ has only trivial units and no zero-divisors (it embeds in a skew-field), it is easy to see that $x$ is invertible in  $\widehat{\mathbb Q G}^\phi$ if and only if there is a unique element in $\mathrm{supp} \, x$ that is the minimum of $\phi$ restricted to the support.

One can see it more geometrically: we can push the support to a finite subset of $G^\mathrm{fab}$ using $\alpha$, and form a polytope by taking the convex hull of this set inside of $G^\mathrm{fab} \otimes_\mathbb Z \mathbb R$. Now $x$ is invertible in  $\widehat{\mathbb Q G}^\phi$ if and only if $\phi$ restricted to the polytope attains its minimum at a unique vertex, and the summand of $x$ that corresponds to this vertex is supported on a singleton, or equivalently, is a trivial unit. (This construction is actually related to that of the Thurston polytope \cite{Thurston1986}, see \cite{Kielak2020} for details.)

What are the possibilities? If we are lucky, we find $\phi$ as above, and $x$ is invertible in  $\widehat{\mathbb Q G}^\phi$. In fact, from the geometric perspectives we see immediately that this is an open property, when we think of $\phi$ as an element of $\mathrm{H}^1(G;\mathbb{R})$.

Otherwise, we can be unlucky in two ways. It could happen that $\alpha(\mathrm{supp} \, x)$ is a singleton. Then, up to multiplication by a group element, $x$ actually lies in the group ring of $\ker \alpha$. The RFRS property gives us a finite-index subgroup $G_1$ of $G$ such that $x \in \mathbb Q G_1$ and the support of $x$ mapped to the free part of the abelianisation of $G_1$  is not a singleton. We might get lucky now, and find $\phi \colon G_1 \to \mathbb{Z}$ such that $x$ is invertible in  $\widehat{\mathbb Q G_1}^\phi$.

Finally, the generic case: there is a non-trivial decomposition $x = \Sigma_{i=1}^k x_i$ such that every $x_i \in \mathbb Q G$ is supported on a single coset of $\ker \alpha$, and the cosets are distinct for different values of $i$. Without loss of generality, $\alpha( \mathrm{supp} \, x_1)$ will be a point on which some $\phi_1 \colon G \to \mathbb Z$ attains its unique minimum. Since our decomposition of $x$ was non-trivial, $x_1$ has smaller support than $x$, and we can proceed by induction. We will find a finite-index subgroup $G_1$ of $G$, and $\phi \colon G_1 \to \mathbb{Z}$ such that $x_1$ is invertible in  $\widehat{\mathbb Q G_1}^\phi$. This is helpful, but we are after invertibility of $x$, rather than $x_1$. Here comes the most technical part of the construction: it can be shown that we may pick $\phi$ lying sufficiently close to $\phi_1|_{G_1}$ such that $x_1$ is invertible in  $\widehat{\mathbb Q G}^\psi$ for every $\psi$ obtained from $\phi$ by conjugating it by the elements of the finite group $G/G_1$. We now form a ring $\widehat{\mathbb Q G_1}^U$ where $U$ is the set of all the maps $\psi$. This ring can be thought of as an intersection of the various Novikov rings (this is a sticky point, we will come back to it). 

Now, the invariance of $U$ under the conjugation action of $G/G_1$ is needed to endow the $\widehat{\mathbb Q G_1}^U$-module $\bigoplus_{G/G_1} \widehat{\mathbb Q G_1}^U$ with a suitable ring structure. The obvious naive choice would be to consider  $\bigoplus_{G/G_1} \widehat{\mathbb Q G_1}^U$ as a sum of rings. We want  however to pick a ring structure that when restricted to $\bigoplus_{G/G_1} \mathbb Q G_1$ would naturally give us $\mathbb Q G$. Since $U$ is invariant under the action of $G/G_1$, this can be done (rather easily), and
we denote the resulting ring by $\widehat{\mathbb Q G_1}^U G/G_1$. Again, by having taken $\phi$ sufficiently close to $\phi_1|_{G_1}$ we make sure that $x$ itself is actually invertible in $\widehat{\mathbb Q G_1}^U G/G_1$.

The upshot: every non-zero element of $\mathbb Q G$ can be inverted in some $\widehat{\mathbb Q G_1}^U G/G_1$. The union of such rings forms a larger ring, and using similar reasoning one can show that every element of $\mathcal D_{\mathbb Q G}$ lies in it.
What is the use of this construction? Take a finite $k\times l$ matrix $A$ over $\mathcal D_{\mathbb Q G}$. We apply our construction to all the entries of the matrix $A$. We obtain a finite index subgroup $G_1$ and an epimorphism $\phi \colon G_1 \to \mathbb Z$. Recall that $\mathcal D_{\mathbb Q G} = \bigoplus_{G/G_1} \mathcal D_{\mathbb Q G_1}$. The matrix $A$ is naturally a $\mathcal D_{\mathbb Q G}$-linear transformation $\bigoplus_k \mathcal D_{\mathbb Q G} \to \bigoplus_l \mathcal D_{\mathbb Q G}$;  we may restrict scalars to $\mathcal D_{\mathbb Q G_1}$, and $A$ becomes a $k|G:G_1|\times l|G:G_1|$ matrix $A'$ over $\mathcal D_{\mathbb Q G_1}$. By our construction, the entries of $A'$ can be viewed as elements of $\widehat {\mathbb Q G}^\phi$.

\begin{thm}[\cite{Kielak2020a,HughesKielak2024}]
	\label{Kielak}
	Let $G$ be a RFRS group of type $\mathtt{FP}_n(\mathbb Q)$. The following are equivalent:
	\begin{enumerate}
		\item $\beta_i^{(2)}(G) = 0$ for all $i \leqslant n$;
		\item there exists a finite index subgroup $G_1 \leqslant G$ and an epimorphism $\phi \colon G_1 \to \mathbb Z$ such that
		$
		\mathrm{H}_i(G;\widehat{\mathbb Q G_1}^\phi) = 0
		$
		for all $i \leqslant n$.
	\end{enumerate}
\end{thm}
\begin{proof}[Sketch proof] That 2.\ implies 1.\ is the main result of \cite{HughesKielak2024}, and does need any assumptions on the group $G$. We will not discuss this in detail, since it is not the direction of interest in this survey.
	
	Suppose that 1.\ holds. We take a free resolution $C_\bullet$ of the trivial $\mathbb Q G$-module $\mathbb Q$ in which the modules $C_i$ are finitely generated for $i \leqslant n$. We endow the modules $C_i$ with bases, and treat the differentials $\partial_i$ as matrices over $\mathbb Q G$. The fact that  $
	\mathrm{H}_i(G;\mathcal D_{\mathbb Q G}) = 0
	$
	for all $i \leqslant n$ is equivalent to the existence of partial chain contractions, that is, we have finite matrices $A_0, \dots, A_n$ over $\mathcal D_{\mathbb Q G}$ such that
	\[
	\partial_{i} A_{i-1} + A_{i} \partial_{i+1} = \mathrm{I}  
	\]
	for all $i \leqslant n$, with $A_{-1}$ being the zero matrix.
	
	By our discussion before, we may pass to a finite-index subgroup $G_1$, keep the resolution $C_\bullet$ (formally, we are restricting the modules to $\mathbb Q G_1$), and replace the matrices $A_i$ by $A'_i$ over $\mathcal D_{\mathbb Q G_1}$. The entries of these matrices can also be thought of as lying over the Novikov ring $\widehat{ \mathbb Q G_1}^\phi$ for some $\phi$, and having partial chain contractions over the Novikov ring immediately gives 
		$
	\mathrm{H}_i(G;\widehat{\mathbb Q G_1}^\phi) = 0
	$
	for all $i \leqslant n$. 
\end{proof}

By passing to a deeper finite-index subgroup if necessary, one can simultaneously take care of Novikov homology with respect to $\phi$ and $-\phi$.

Two points from the above outline need further discussion. First, the part of the proof where we produce $\phi$ sufficiently close to $\phi_1|_{G_1}$ has been simplified and streamlined significantly by Okun--Schreve \cite{OkunSchreve2025}. Instead of focusing on morphisms to $\mathbb Z$, they look at orderings instead, and it produces a much cleaner proof. More generally, orderings are probably the way to go: already the very definition of BNS invariants focuses on the monoid $G_\phi$ that is defined in terms of an ordering. Also, some of the discussion above can be taken beyond RFRS groups using orderings, see \cite{Klinge2023} and \cite{FisherKlinge2024}.

The second point: how is one to understand $\widehat {\mathbb Q G_1}^U$? It will be instructive to look at a simple example. Take $G_1 = \langle t \rangle \cong \mathbb Z$, and let $x = 1-t$. This element is invertible in $\mathcal D_{\mathbb Q G_1}$, which in this case coincides with the field of rational functions with coefficients in $\mathbb Q$ and single variable $t$. It is also invertible in $\widehat {\mathbb Q G_1}^{\mathrm{id}}$, with inverse $\sum_{i \geqslant 0} t^i$, and over $\widehat {\mathbb Q G_1}^{-\mathrm{id}}$, with inverse $-\sum_{i < 0} t^i$. The element $x$ is not however invertible in $\widehat {\mathbb Q G_1}^{\mathrm{id}} \cap \widehat {\mathbb Q G_1}^{-\mathrm{id}} = \mathbb Q G_1$, so taking the intersection is not what we want.

The good properties of the skew-field $\mathcal D_{\mathbb Q G_1}$ that we discussed above allow us to embed it into the twisted Laurent power series ring with coefficients in $\mathcal D_{\mathbb Q \ker \phi}$ and variable $t$. The ring $\widehat {\mathbb Q G_1}^\phi$ embeds here as well, and we can take the intersection of these two rings. This way, for every $\phi$, we obtain a subring of $\mathcal D_{\mathbb Q G_1}$ consisting of elements that can be realised over the Novikov ring. We define $\widehat {\mathbb Q G_1}^U$ to be the intersection of such rings for all $\phi \in U$.

It is in this last construction that we used the fact that we are working over $\mathbb Q$ -- we need access to $\mathcal D_{\mathbb Q G_1}$ as an ambient ring in which we intersect subrings. In fact, the role of the ambient ring can be played by $\mathcal U(G)$, which is bigger, but again is a characteristic-zero object.

Andrei Jaikin-Zapirain \cite{Jaikin-Zapirain2021} constructed other ambient rings in positive characteristics. This allowed him to run a (slightly simplified) version of the above construction. He used this to \emph{define} skew-fields $\mathcal D_{\mathbb K G}$ for other ground fields $\mathbb K$, when $G$ is a RFRS group. With this new construction, the whole discussion goes through in all characteristics. Using \cref{Sikorav}, and defining
\[
\beta_i^{(2)}(G; \mathbb K) = \dim_{\mathcal D_{\mathbb K G}} \mathrm{H}_i(G;\mathcal D_{\mathbb K G})
\]
we obtain a strengthening of \cref{Kielak}.

\begin{thm}[Fisher \cite{Fisher2024}]
	\label{Fisher}
	Let $\mathbb K$ be a field.
	Let $G$ be a RFRS group of type $\mathtt{FP}_n(\mathbb K)$. The following are equivalent:
	\begin{enumerate}
		\item $\beta_i^{(2)}(G; \mathbb K) = 0$ for all $i \leqslant n$;
		\item there exists a finite index subgroup $G_1 \leqslant G$ that is $\mathtt{FP}_n(\mathbb K)$-fibred.
	\end{enumerate}
\end{thm}

The implication 2. $\Rightarrow$ 1. is true in general, without $G$ being RFRS.

\begin{thm}[Generalised Mapping Torus Theorem \cite{HennekeKielak2021}]
	\label{HennekeKielak}
	Let $\mathbb K$ be a field and
	let $K$ be a group of type $\mathtt{FP}_n(\mathbb K)$. For every semi-direct product $G = K \rtimes \mathbb Z$ we have $\beta_n^{(2)}(G;\mathbb K) = 0$, provided that the number is defined.
\end{thm}

Combining \cref{Fisher} over $\mathbb Q$ with \cref{Stallings} gives a new proof of \cref{LottLueck}, and in particular of \cref{Agol}. The proofs are very different: the new is algebraic whereas Agol's one is topological, but nevertheless there is a curious similarity in the broad outlines of both proofs.

Can we go beyond fields? Not with this method. We rely on the fact that every non-zero group ring element supported on a singleton is invertible, and this stops being true if the coefficients contain non-zero non-units.

What about enlarging the class of groups? It is perhaps not very clear from the above outline, but the proof uses the definition of RFRS in a strong way. On the other hand, the class of non-trivial torsion-free nilpotent groups has very good fibring properties, and the $L^2$-homology of such groups is always trivial, hence they seem to fit naturally into our setup. Unfortunately, the only nilpotent groups that are RFRS are abelian groups, and hence it is tempting to define a larger class of groups, containing both RFRS and torsion-free nilpotent groups. This was attempted by Fisher--Klinge \cite{FisherKlinge2024} -- they study the class of groups that are residually \{virtually nilpotent and poly-$\mathbb Z$\}. They did not manage to get results quite as strong as \cref{Fisher}, but they did get a result on dropping cohomological dimension.

\section{Applications}
\label{sec apps}

\subsection{Dropping dimension}  

When a closed aspherical manifold fibres, the dimension of the fibre is lower than that of the original manifold. There are group-theoretic settings exhibiting a similar phenomenon. Let $\Sigma$ be the fundamental group of a closed surface of genus at least $1$, and take any epimorphism $\phi \colon \Sigma \to \mathbb Z$. The kernel of $\phi$ is always a free group. When the genus is $1$, the kernel is a copy of $\mathbb Z$, and fits right into our setup (the group $\Sigma$ is always RFRS, and it is $L^2$-acyclic if and only if the genus is $1$). When the genus is at least $2$, the kernel is an infinitely generated free group, and so $\Sigma$ fails to be $\mathtt{FP}_1$-fibred. However, the dimension of $\Sigma$ is $2$, but that of $\ker \phi$ is $1$, in the following sense.

\begin{definition}[Cohomological dimension]
	The \emph{cohomological dimension} of a group $G$ over a ring $R$, denoted $\mathrm{cd}_R(G)$, is the infimum of lengths of resolutions of the trivial $RG$-module $R$, where a resolution $C_\bullet$ has \emph{length $n$} if and only if $C_i = 0$ for all $i > n$ and $C_n \neq 0$.  
\end{definition}

It is immediate that cohomological dimension of a subgroup is bounded above by that of the supergroup -- a resolution for the latter is a resolution for the former via restriction.

When the cohomological dimension of $G$ over $R$ is finite, it coincides with with the greatest $n$ such that $\mathrm{H}^n(G;S) \neq 0$ for some $RG$-module $S$. When $G$ is of type $\mathtt{FP}(R)$, we can say more: the cohomological dimension over $R$ coincides with the greatest $n$ such that $\mathrm{H}^n(G;RG) \neq 0$. These descriptions will be very useful for us.

Since we are going to shift focus from homology to cohomology, let us record an observation.

\begin{lem}[Corollary 3.2 of \cite{Fisheretal2025}]
	\label{Novikov cohom}
	For a group $G$, if $\mathrm{H}_i(G; \widehat{ R G}^\phi) = 0$ for all $i\leqslant n$ then $\mathrm{H}^i(G; \widehat{ R G}^{\phi}) = 0$ for all $i\leqslant n$. 
\end{lem}

A quick word on the sign: the fact that we do not need to change the second $\phi$ above to $-\phi$ comes from our notational convention -- coefficients of homology are right modules, whereas those of cohomology are left modules. If one follows Brown and uses left modules everywhere, the sign needs to change.

We do not have an interpretation of the vanishing of Novikov cohomology in terms of properties of the group in general, but we do understand what happens in the top dimension.

\begin{thm}[Fisher, Theorem D in \cite{Fisher2024a}]
	\label{drop}
	Let $G$ be a group of type $\mathtt{FP}(R)$ with $n =  \mathrm{cd}_R(G)$. If $\phi \colon G \to \mathbb Z$ is an epimorphism with
	\[
	\mathrm{H}^n(G;\widehat{R G}^\phi) =  	\mathrm{H}^n(G;\widehat{R G}^{-\phi}) = 0
	\]
	then $\mathrm{cd}_R(\ker \phi) = n-1$.
\end{thm} 
\begin{proof}[Sketch proof]
	We know that $\mathrm{cd}_R(\ker \phi) \leqslant n$. Using the description above, we only need to check that $\mathrm{H}^n(\ker \phi;S) = 0$ for every $R\ker \phi$-module $S$. In this sketch we are going to do this only for $S = R\ker \phi$, since it exemplifies the idea well, but is notationally less cumbersome.
	
Let $t \in G$ be such that $\phi(t) = 1$, as before. 	Consider $R \ker \phi$ as a module over itself.  Out of it, we can obtain the $RG$-module $RG$ by induction, but also the $RG$-module $\prod_{i \in \mathbb Z} (R \ker\phi) t^i$ by coinduction. The Novikov rings relate these two modules -- we have the following short exact sequence
\[
0 \to RG \to \widehat{R G}^\phi \oplus \widehat{R G}^{-\phi} \to \prod_{i \in \mathbb Z} (R \ker\phi) t^i \to 0
\] 
where the first map is the diagonal embedding, and the second is the obvious embedding on the first factor, and minus the obvious embedding on the second. Out of this short exact sequence we get a long exact sequence in cohomology
\begin{equation}
		 \tag{$\dagger$}
		 \label{les}
\begin{aligned}
	\begin{adjustbox}{valign=c,scale=0.75}
	$\cdots \to \mathrm{H}^n(G;\widehat{R G}^\phi) \oplus \mathrm{H}^n(G;\widehat{R G}^{-\phi}) \to \mathrm{H}^n(G;\prod_{i \in \mathbb Z} (R \ker\phi) t^i)  \to \mathrm{H}^{n+1}(G;R G) \to \cdots . $
	\end{adjustbox}
\end{aligned}	
\end{equation}

The two outside terms vanish: the first by assumption, the second by $\mathrm{cd}_R(G) = n$. The middle term by Shapiro's lemma is isomorphic to $\mathrm{H}^n(\ker \phi; R \ker\phi)$.

When dealing with a general $R\ker \phi$-module $S$, the condition that $G$ is of type  $\mathtt{FP}(R)$ becomes important.
\end{proof}

Back to the example of the surface group $\Sigma$ -- we had a weak form of fibring, since we have maps $\phi$ whose kernel has strictly lower cohomological dimension. At the same time, the top-dimensional $L^2$-homology of $\Sigma$ is zero.

\begin{thm}[Fisher, Theorem E in \cite{Fisher2024a}]
	Let $G$ be a RFRS group of type $\mathtt{FP}_{n-1}(\mathbb K)$, where $\mathbb K$ is a field, and suppose that $\mathrm{cd}_{\mathbb K}(G) = n$. The following are equivalent:
	\begin{enumerate}
	\item $\beta_n^{(2)}(G; \mathbb K) = 0$;
	\item there exists a finite index subgroup $G_1 \leqslant G$ and an epimorphism $\phi \colon G_1 \to \mathbb{Z}$ with $\mathrm{cd}_{\mathbb K}(\ker \phi) = n-1$.
\end{enumerate}	
\end{thm}
\begin{proof}[Sketch proof]
Here is a very rough outline of the argument. We start by taking a projective resolution $C_\bullet$ of the trivial $RG$-module $R$ that has length $n$, and in which $C_{n-1}$ is free and finitely generated. Now, the vanishing of the top-dimensional $L^2$-homology tells us that the top differential is injective after tensoring with $\mathcal D_{\mathbb K G}$. This immediately implies that	it admits a right inverse over $\mathcal D_{\mathbb K G}$. Using the usual yoga for RFRS groups, after passing to a finite-index subgroup $G_1$ we may find a right inverse in a Novikov ring $\widehat{ R G_1}^\phi$ for some epimorphism $\phi \colon G_1 \to \mathbb Z$ and for its negative. But then the codifferential admits a left inverse over the Novikov rings, and hence is onto. This kills top-dimensional Novikov cohomology, and we conclude by applying \cref{drop}. 
\end{proof}

This result has a bit of a history. The case $n=2$ was first proved under stronger assumptions of $G$ being hyperbolic and compact special in \cite{KielakLinton2023a}. In this case, $n-1=1$ and, using the Stallings--Swan theorem, one can conclude that the kernel $\ker \phi$ is in fact free, making $G$ into a virtually free-by-$\mathbb Z$ group (in fact, we need Dunwoody's work \cite{Dunwoody1979} here). This was used to resolve a conjecture of Baumslag -- all one-relator groups with torsion are virtually free-by-cyclic.

As mentioned before, Fisher--Klinge  obtained similar results for groups that are residually \{virtually nilpotent and poly-$\mathbb Z$\}, see Theorem A in \cite{FisherKlinge2024}.

\subsection{Coherence}

Algebraic fibring is useful not only in geometric topology, but also in group theory.

\begin{definition}[Coherence]
	\begin{enumerate}
		\item We say that a group $G$ is \emph{coherent} if and only if every finitely generated subgroup of $G$ is finitely presented.
		\item We say that a group $G$ is \emph{homologically coherent} if and only if every finitely generated subgroup of $G$ is of type $\mathtt{FP}_2$.
		\item We say that a ring $R$ is \emph{coherent} if and only if every finitely generated left $R$-module is finitely presented. 
	\end{enumerate}
\end{definition}	

It is clear that coherence implies homological coherence,  and so does coherence of the group ring $\mathbb Z G$. The implications are not known to be proper in either case.

Coherence is considered to be a property of (some) low-dimensional groups. Indeed, it is known for free and surface groups, for $3$-manifold groups (due to a celebrated theorem of Scott \cite{Scott1973}), for free-by-cyclic groups (by Feighn--Handel \cite{FeighnHandel1999}), and now also for one-relator groups, thanks to the recent breakthrough of Jaikin-Zapirain--Linton \cite{Jaikin-ZapirainLinton2023}. All of these groups have cohomological dimension at most $2$ over the rationals (well, fundamental groups of closed $3$-manifolds obviously do not, but because of Poincar\'e duality, they tend to behave as if they did).

The most important example of an incoherent group is the product of two free group $F_2 \times F_2$. Picking free bases $\{a,b\}$ and $\{c,d\}$ for the factors, we build an epimorphism $\phi \colon F_2 \times F_2 \to \mathbb Z$ by sending each of the chosen  generators to $1$. It can be shown directly that the kernel is finitely generated -- indeed, $\{ ac^{-1}, ad^{-1}, bc^{-1}, bd^{-1} \}$ generates $\ker \phi$. The kernel is not however of type $\mathtt{FP}_2$ -- to stay true to the spirit of the survey, we can compute $\beta_2^{(2)}(F_2 \times F_2) = 1$ using a Meyer--Vietoris sequence, and then appeal to  \cref{HennekeKielak}. One can also see it more directly by computing the second homology of $\ker \phi$ with trivial coefficients, and checking that it is not finitely generated, which cannot happen for groups of type $\mathtt{FP}_2$.

Since coherence passes to subgroups, $F_2 \times F_2$ is often used as a poison subgroup in this context. More generally, one can use as a poison subgroup any group $G$ that is RFRS and has $\beta_1^{(2)}(G) = 0$ but $\beta_2^{(2)}(G) \neq 0$, as such a group will contain a subgroup (the kernel of a virtual epimorphism to $\mathbb Z$) that is finitely generated but not of type $\mathtt{FP}_2(\mathbb Q)$.

More generally, this technique can be used to look into \emph{higher coherence} -- situations in which every subgroup of type $\mathtt{FP}_n$ is automatically of type $\mathtt{FP}_{n+1}$. As before, containing a subgroup $G$ that is RFRS and has all $L^2$-Betti numbers zero except in dimension $n+1$ provides a witness for the lack of such higher coherence. This was used for example by Llosa Isenrich--Martelli--Py \cite{Llosa-Isenrichetal2024}, see also \cite{Fisher2024,Kudlinska2023}.

\section{Higher dimensions}
\label{sec higher dim}

\subsection{Poincar\'e-duality groups}

In the previous section we looked at a weaker notion of fibring for groups, when it was enough for the cohomological dimension to drop. Now we are going to look at a stronger one, much closer to the original motivation coming from aspherical manifolds.

\begin{definition}
	A group $G$ is said to be a \emph{Poincar\'e-duality group} over $R$ of dimension $n$, written $\mathrm{PD}^n(R)$, if and only if $G$ is of type $\mathtt{FP}(R)$ and 
	\[
	\mathrm{H}^i(G; RG) \cong \left\{ \begin{array}{cl} R & \textrm{ if } i=n \\ 0 & \textrm{ otherwise} \end{array} \right.
	\]
	as $R$-modules.

The module $\mathrm{H}^n(G; RG)$ has a natural left $RG$-action, known as the \emph{orientation action}. The $\mathrm{PD}^n(R)$-group is \emph{orientable} if and only if the orientation action is trivial.
\end{definition}

As for finiteness properties, one writes $\mathrm{PD}^n$ for $\mathrm{PD}^n(\mathbb Z)$, and again this property implies $\mathrm{PD}^n(R)$ over all rings.

One can also define $\mathrm{PD}^n(R)$-pairs that mimic the behaviour of pairs $(M,\partial M)$ of a manifold $M$ with boundary $\partial M$. 

It might not be immediately clear from the (succinct) definition, but orientable $\mathrm{PD}^n(R)$-groups exhibit duality between homology and cohomology. To see this, take a finite resolution $C_\bullet$ of length $n$ of the trivial $RG$-module $R$ by finitely generated projective modules. We can dualise it, that is, consider the cochain complex with terms $C^i = \mathrm{Hom}_{RG}(C_i, RG)$. Since every $C_i$ is finitely generated and projective, so is every $C^i$. Now, the cohomology of the dual cochain complex computes precisely $\mathrm{H}^*(G; RG)$, and our assumption tells us that this cochain complex is exact everywhere, the first codifferential is injective, and the cokernel of the last is the trivial $RG$-module $R$. Hence, reversing the direction of the arrows, the chain complex $(C^{n-i})_i$ is a resolution of $R$. Thus, computing the $i$th homology is the same as computing the $(n-i)$th cohomology.

In the non-orientable case, homology and cohomology are still dual, but there is an additional twisting due to tensoring with the module $\mathrm{H}^n(G; RG)$.

Vice-versa, having duality between homology and cohomology (and the right finiteness properties) allows one to recover the definition above: homology with coefficients in $RG$ is trivial in all dimensions except zero, where it is $R$, by the very definition of group homology.

Since closed aspherical manifolds exhibit this duality, their fundamental groups are obvious examples of Poincar\'e-duality groups (over $\mathbb Z$). There are no known examples of $\mathrm{PD}^n$-groups that are finitely presented but not fundamental groups of aspherical manifolds. If one drops the requirement of the group being finitely presented, we have uncountably many examples, but all in dimension at least $4$, by the work of Davis \cite{Davis2000} and Leary \cite{Leary2018}. It is not known if there are such examples in dimension $3$.

\begin{thm}[Eckmann--Linnell--M\"uller \cite{EckmannMueller1980,EckmannLinnell1982,EckmannLinnell1983}, Bowditch \cite{Bowditch2004}, see also \cite{KielakKropholler2021}]
	\label{PD2}
	If $G$ is a $\mathrm{PD}^2(\mathbb K)$-group and $\mathbb K$ is a field, then $G$ is virtually a surface group. If $\mathbb K = \mathbb Z$, then $G$ is a surface group.
\end{thm}

In the context of Poincar\'e-duality groups, virtual fibring should work exactly the same as for closed aspherical manifolds.

\begin{conj}
	\label{PD conj}
Let $G$ be a $\mathrm{PD}^n$-group that is RFRS. The following are equivalent:
\begin{enumerate}
	\item $\beta_i^{(2)}(G; \mathbb K) = 0$ for all $i$ and all fields $\mathbb K$;
	\item there exists a finite index subgroup $G_1 \leqslant G$ and an epimorphism $\phi \colon G_1 \to \mathbb{Z}$ such that $\ker \phi$ is a $\mathrm{PD}^{n-1}$-group.
\end{enumerate}
\end{conj}

There is actually strong evidence towards this conjecture.

\begin{thm}[\cite{Fisheretal2025}]
	\label{PD thm}
	Let $G$ be a $\mathrm{PD}^n$-group that is RFRS. The following are equivalent:
	\begin{enumerate}
		\item $\beta_i^{(2)}(G; \mathbb K) = 0$ for all $i$ and all fields $\mathbb K$;
		\item there exists a finite index subgroup $G_1 \leqslant G$ and an epimorphism $\phi \colon G_1 \to \mathbb{Z}$ such that $\ker \phi$ is an orientable  $\mathrm{PD}^{n-1}(\mathbb K)$-group for every field $\mathbb K$.
	\end{enumerate}
\end{thm}

The difficulty in upgrading \cref{PD thm} to \cref{PD conj} lies in the fact that being of type $\mathtt{FP}_n(\mathbb K)$ over every field $\mathbb K$ does not imply being of type $\mathtt{FP}_n$. One of Abels's groups can be used to build a counterexample, as explained by Bieri--Strebel \cite{BieriStrebel1980}. There is also a new counterexample in the forthcoming work of Robert Kropholler. The Bieri--Strebel example actually appears as a kernel of an epimorphism to $\mathbb Z$ from a finitely presented group. The problem seems subtle, and its resolution will probably require heavy use of Poincar\'e duality.

If one zooms in at one field at the time, a version of \cref{PD conj} is true.

\begin{thm}[\cite{Fisheretal2025}]
	\label{PD single K thm}
	Let $\mathbb K$ be a field.
	Let $G$ be an orientable $\mathrm{PD}^n(\mathbb K)$-group that is RFRS. The following are equivalent:
	\begin{enumerate}
		\item $\beta_i^{(2)}(G; \mathbb K) = 0$ for all $i$;
		\item there exists a finite index subgroup $G_1 \leqslant G$ and an epimorphism $\phi \colon G_1 \to \mathbb{Z}$ such that $\ker \phi$ is an orientable $\mathrm{PD}^{n-1}(\mathbb K)$-group.
	\end{enumerate}
\end{thm}
\begin{proof}[Sketch proof]
	The implication 2.\ $\Rightarrow$ 1.\ follows from \cref{HennekeKielak}.
	
Suppose that 1.\ holds. By the usual argument and \cref{Novikov cohom}, we obtain a finite-index subgroup $G_1$ and an epimorphism $\phi \colon G_1 \to \mathbb{Z}$ such that
\[
\mathrm{H}_i(G_1; \widehat{ \mathbb K G_1}^{\pm \phi}) = \mathrm{H}^i(G_1; \widehat{ \mathbb K G_1}^{\pm \phi})=0
\]
for all $i$. By \cref{Sikorav}, the first vanishing tells us that $\ker \phi$ is of type $\mathtt{FP}(\mathbb K)$ (there is a little argument needed here, using the bound $\mathrm{cd}_{\mathbb K} (\ker \phi) \leqslant n$, see Proposition VIII.6.1 in \cite{Brown1982}).

Recall the long exact sequence \eqref{les}:
\begin{equation*}
\begin{adjustbox}{valign=c,scale=0.67}
	$ \cdots \to \mathrm{H}^{i}(G_1;\mathbb K G_1) \to \mathrm{H}^i(G_1;\widehat{\mathbb K G_1}^{\phi}) \oplus \mathrm{H}^i(G_1;\widehat{\mathbb K G_1}^{-\phi})  \\ \to \mathrm{H}^i(G_1;\prod_{j \in \mathbb Z} (\mathbb K \ker\phi) t^j) \to \mathrm{H}^{i+1}(G_1;\mathbb K G_1) \to \cdots . $
\end{adjustbox}
\end{equation*}
The Novikov cohomology terms vanish, and so \[\mathrm{H}^{i+1}(G_1;\mathbb K G_1) \cong \mathrm{H}^i(G_1;\prod_{j \in \mathbb Z} (\mathbb K \ker\phi) t^j)\] for all $i$. But the latter cohomology, by Shapiro's lemma, is precisely $\mathrm{H}^i(\ker \phi; \mathbb K \ker\phi)$, and therefore if $G_1$ is $\mathrm{PD}^n(\mathbb K)$ then $\ker \phi$ is $\mathrm{PD}^{n-1}(\mathbb K)$. The fact that $G_1$, being finite-index in $G$, is $\mathrm{PD}^n(\mathbb K)$ is easy to establish: finiteness properties pass to finite-index subgroups, and the cohomology with group ring coefficients of $G_1$ is equal to the corresponding cohomology for $G$ by Shapiro's lemma.
\end{proof}

It was known before that for kernels of epimorphisms to $\mathbb Z$ from Poincar\'e-duality groups it is enough to establish sufficiently good finiteness properties to conclude that the kernel is a Poincar\'e-duality group itself, see for example the work of Hillman--Kochloukova \cite{HillmanKochloukova2007}. The proof given above is however new and very straightforward.

The argument combined with \cref{PD2} gives a quick proof of \cref{Stallings}: we start with a $3$-manifold $M$ with fundamental group $G$ and an epimorphism $\phi \colon G \to \mathbb Z$ with finitely generated kernel. A little yoga allows us to focus on the case of $M$ being aspherical. From \cref{BNS,Sikorav} we get vanishing of Novikov homology over $\mathbb Z$ in dimension $1$; vanishing in dimension $0$ holds for all Novikov rings, as long as $\phi$ is not trivial. \Cref{Novikov cohom} gives us vanishing of Novikov cohomology in dimensions $0$ and $1$, and now Poincar\'e duality gives vanishing of all Novikov homology. Arguing as above, we show that $\ker \phi$ is a $\mathrm{PD}^2$-group, and so a surface group by \cref{PD2}. Every automorphism of a surface group can be realised by a mapping class (this is the Dehn--Nielsen--Baer theorem), and we are done.

\subsection{Hyperbolic manifolds}

We have made our hopes regarding Poin\-car\'e-duality groups explicit in \cref{PD conj}. We have similar expectations in the realm of manifolds. 

\begin{conj}
	\label{manifold conj}
	Let $M$ be a closed aspherical manifold whose fundamental group $G$ is RFRS. The following are equivalent:
	\begin{enumerate}
		\item $\beta_i^{(2)}(G; \mathbb K) = 0$ for all $i$ and all fields $\mathbb K$;
		\item $M$ is virtually fibred.
	\end{enumerate}
\end{conj}

In view of \cref{Farrell}, for high-dimensional aspherical  manifolds with vanishing Whitehead group, the key difference between \cref{PD conj} and \cref{manifold conj} lies in finite presentability of $\ker \phi$. This is not a homological property, and establishing it will need new ideas. Perhaps another invariant will need to be added to the list of known obstructions.

For specific examples, finite presentability of $\ker \phi$ can be established using the Jankiewicz--Norin--Wise combinatorial game \cite{Jankiewiczetal2017}. This method applies to right-angled Coxeter groups and has been used in a crucial way in the following.

\begin{thm}[Italiano--Martelli--Migliorini \cite{Italianoetal2023}]
	\label{IMM}
	There exists a $5$-dimensional non-compact hyperbolic manifold of finite volume that fibres.
\end{thm}

By a \emph{hyperbolic manifold} we will mean a manifold obtained by taking a quotient of the real-hyperbolic space by a group of isometries that acts freely and properly discontinuously.

Together with the fact that all finite-volume hyperbolic manifolds in dimension $3$  virtually fibre, this is evidence towards the following audacious statement.

\begin{conj}
	\label{audacious conj}
	All finite-volume hyperbolic manifolds in odd dimensions are virtually fibred.
\end{conj}

As discussed before, the restriction to odd dimensions comes from the Euler characteristic. The other type of obstruction, $L^2$-homology, does not add anything in this context.

\begin{thm}[Dodziuk \cite{Dodziuk1979}]
	\label{Dodziuk}
	Every odd-dimensional closed hyperbolic manifold is $L^2$-acyclic.
\end{thm}

This can be extended to all finite-volume hyperbolic manifolds using measure equivalence and the work of Gaboriau \cite{Gaboriau2002}. What happens over fields other than $\mathbb Q$?

\begin{conj}
	\label{Dodziuk conjecture}
	If $G$ is the fundamental group of a finite-volume hyperbolic odd-dimensional manifold with RFRS fundamental group $G$, then $\beta_i^{(2)}(G;\mathbb K) = 0$ for every $i$ and every field $\mathbb K$.
\end{conj}

In view of \cref{Dodziuk,Dodziuk conjecture},  \cref{audacious conj} is not obviously wrong. The evidence for it is rather scant: it holds in dimension $3$, there is one example (up to commensuration) in dimension $5$, and that is it. Topologically, a fibration is a special kind of a circle-valued Morse function -- one without any critical points. The difference between the number of critical points of even index and the number of those of odd index computes the Euler characteristic. One says that a circle-valued Morse function is \emph{perfect} if and only if the number of critical points is equal to the absolute value of the Euler characteristic of the manifold. Let us introduce a stronger condition -- we will call such a Morse function \emph{Singer} if and only if all the critical points have index equal to half the dimension of the manifold. So, for an odd-dimensional manifold, the existence of a Singer circle-valued Morse function is equivalent to fibring. (Two words on the terminology: it is not standard -- in fact, it appears here for the first time! It is motivated by the Singer conjecture that predicts that closed aspherical manifolds should have at most one non-trivial $L^2$-Betti number, exactly in the middle dimension.) With this notation, \cref{audacious conj} is a special case of the following statement: every finite-volume hyperbolic manifold should virtually admit a Singer circle-valued Morse function. And this statement has additional supporting evidence -- it holds in dimensions $2$ and $3$, and there are many examples of hyperbolic manifolds that admit such functions in dimensions $4$ and $6$ (both closed and cusped), see the constructions of Battista--Martelli \cite{BattistaMartelli2021} and Italiano--Migliorini \cite{ItalianoMigliorini2025}. Moreover, going away from hyperbolic manifolds, for arithmetic closed complex-hyperbolic manifolds with positive first (usual) Betti number the statement is true!  This was shown by Llosa Isenrich--Py, see Corollary 8.5 in \cite{Llosa-IsenrichPy2025}.

Here is a more believable special case of \cref{audacious conj}.

\begin{conj}
	\label{RFRS conj}
	All finite-volume hyperbolic manifolds in odd dimensions with RFRS fundamental groups are virtually fibred.
\end{conj}

It is more believable, because we have a strategy to prove it in three steps:

\noindent \textbf{Step 1:} Prove \cref{Dodziuk conjecture}. Dodziuk's proof of \cref{Dodziuk} is geometric, and it is not clear how to translate it into an algebraic one. This step is particularly interesting, since if one can find a counterexample here, it will automatically be the first example of a hyperbolic manifold in odd dimension that does not virtually fibre.

\noindent \textbf{Step 2:} Prove \cref{PD conj}. In conjunction with \cref{PD thm} and the previous step, it would give us a virtual epimorphism to $\mathbb Z$ whose kernel is a Poincar\'e-duality group.

\noindent \textbf{Step 3:} Prove that the kernel is finitely presented. Then we may apply \cref{Farrell aspherical} -- our $G$ is either hyperbolic (in the sense of Gromov), or hyperbolic relative to free-abelian groups. In either case we have the Farrell--Jones conjecture for $G$ by \cite{Bartelsetal2007,Bartels2017}, and so $\mathrm{Wh}(G) = 0$.

Step 3 will probably end up being the most difficult one. It will require a more homotopic approach.

Dealing with the first two steps would immediately give the first example of a fibred finite-volume hyperbolic manifold in dimension $7$, since one can get finitely generated kernels using \cite{Italianoetal2022}.

There is also potentially bad news for \cref{RFRS conj}: Avramidi--Okun--Schreve \cite{Avramidietal2023} constructed a closed aspherical manifold in dimension $7$ whose fundamental group is Gromov-hyperbolic and RFRS, and which does not virtually fibre. This does not contradict the conjecture, since the manifold is not known to be hyperbolic. It also does not contradict  \cref{manifold conj} -- the impossibility of virtual fibring is shown precisely by exhibiting a non-vanishing $L^2$-Betti number, in positive odd characteristic.

How far are \cref{audacious conj,RFRS conj} from one another? We do not know any examples of finite-volume hyperbolic manifolds whose fundamental group is not virtually RFRS. The compact arithmetic hyperbolic manifolds of simple type are all virtually RFRS by Bergeron--Wise \cite{BergeronWise2012}, and this is a very rich source of examples.

Having a fibred hyperbolic manifold beyond dimension $3$ opens up many immediate questions: what geometry (if any) does the fibre carry? It is certainly not hyperbolic due to Paulin's theorem \cite{Paulin1991}. What dynamical properties does the monodromy have? Will we have a pseudo-Anosov type behaviour?

\section*{Acknowledgments}
This survey is based on articles written with Martin Bridson, Sam Fisher, Florian Funke, Fabian Henneke, Sam Hughes, Giovanni Italiano, and Monika Kudlinska. It also draws heavily from the work of Ian Agol, Stefan Friedl, Andrei Jaikin-Zapirain, Peter Linnell, and Wolfgang L\"uck. Additionally, the author has benefited greatly from conversations with all these people, with the sad exception of Peter Linnell, whose ideas remain the key inspiration.

\bibliographystyle{siamplain}
\bibliography{bibliography}
\end{document}